\documentclass[11pt]{article}
\usepackage{amssymb, amsmath, amsthm, graphicx, hyperref}
\usepackage[left=1in,top=1in,right=1in]{geometry}
\date{}
\allowdisplaybreaks

\usepackage[colorinlistoftodos]{todonotes}

\theoremstyle{definition}
\newtheorem{theorem}{Theorem}
\newtheorem{corollary}[theorem]{Corollary}
\newtheorem{claim}[theorem]{Claim}

\title{To cover a permutohedron}
\author{Bochao Kong\thanks{Michigan State University, East Lansing, MI, USA. Email: {\tt kongboc1@msu.edu}.} \and Ji Zeng\thanks{Alfréd Rényi Institute of Mathematics, Budapest, Hungary. Supported by ERC Advanced Grants ``GeoScape'', no. 882971 and ``ERMiD'', no. 101054936. Email: {\tt zeng.ji@renyi.hu}.}}

\begin{document}

\maketitle

\begin{abstract}
  The permutohedron $P_n$ of order $n$ is a polytope embedded in $\mathbb{R}^n$ whose vertex coordinates are permutations of the first $n$ natural numbers. It is obvious that $P_n$ lies on the hyperplane $H_n$ consisting of points whose coordinates sum up to $n(n+1)/2$. We prove that if the vertices of $P_n$ are contained in the union of $m$ affine hyperplanes different from $H_n$, then $m\geq n$ when $n \geq 3$ is odd, and $m \geq n-1$ when $n \geq 4$ is even. This result has been established by Pawlowski in a more general form. Our proof is shorter, rather different, and gives an algebraic criterion for a non-standard permutohedron generated by $n$ distinct real numbers to require at least $n$ non-trivial hyperplanes to cover its vertices.
\end{abstract}

% 52B05

Let $A$ be a set of $n$ distinct real numbers. We use $P_A$ to denote the polytope in $\mathbb{R}^n$ whose vertex coordinates are permutations of $A$. It is easy to argue that all these points, whose coordinates are permutations of $A$, are in convex position. It is also obvious that $P_A$ is contained in the hyperplane $H_A$ defined by the equation $\sum_{j\in [n]} x_j = \sum_{a\in A} a$. Here, $[n] := \{1,2,\dots,n\}$ and $x_j$ is the $j$-th coordinate. In fact, $P_A$ is always $(n-1)$-dimensional, see, e.g. \cite{ziegler1995lectures}. The special case $P_{[n]}$ is known as the \textit{permutohedron} of order $n$. We write $P_n := P_{[n]}$ and $H_n := H_{[n]}$ for simplicity.

A collection $\mathcal{C}$ of affine hyperplanes is called a \textit{vertex cover} of $P_A$ if $H_A \not\in \mathcal{C}$ and every vertex of $P_A$ lies on some hyperplane in $\mathcal{C}$. It is obvious that $P_A$ can always be covered by $n$ hyperplanes defined by equations $x_1 = a$ for $a \in A$. However, when $n$ is even, there are $n-1$ hyperplanes, defined by $x_1+x_j = n+1$ for $j\in [n]\setminus 1$, that contain all vertices of $P_n$. Recently, Hegedüs and Károlyi~\cite{hegedus2024covering} conjectured the following statement, which is our main result.

\begin{theorem}\label{thm:main}
  If $n\ge 3$ is odd, then every affine hyperplane $H\subset\mathbb{R}^n$ with $H\neq H_n$ contains at most $(n-1)!$ vertices of $P_n$.
\end{theorem}

\begin{corollary}\label{main}
  If $n \geq 3$ is odd, a vertex cover of $P_n$ must have size at least $n$. If $n \geq 4$ is even, a vertex cover of $P_n$ must have size at least $n-1$.
\end{corollary}

\begin{proof}[Proof of Corollary~\ref{main}]
  The statement for $n\ge 3$ odd is an immediate consequence of Theorem~\ref{thm:main} by counting. When $n\ge 4$ is even, the bound $n-1$ follows from the odd-dimensional case and the reduction argument in \cite{hegedus2024covering} (paragraph after Conjecture~6).
\end{proof}

After circulating an earlier draft, we learned that Pawlowski~\cite{pawlowski2024} recently proved a general result akin to Theorem~\ref{thm:main} with $P_n$ replaced by $P_A$, and it answers an earlier conjecture by Huang, McKinnon, and Satriano~\cite{huang2021fraction}. The proof in \cite{pawlowski2024} relied on the Bruhat order and the Sperner property. The authors in~\cite{huang2021fraction} proved their conjecture in some special cases by an analysis via algebraic geometry albeit not concluding Theorem~\ref{thm:main}. We shall analyze the same variety as in \cite{huang2021fraction} rather differently and obtain a shorter proof of Theorem~\ref{thm:main}. Our proof gives an algebraic criterion on $A$ ensuring that any non-trivial hyperplane contains at most $(n-1)!$ vertices of $P_A$. We write the elementary symmetric polynomial on $n$ variables of degree $d$ as
\begin{equation*}
  S_d(\textbf{x}) = S_d(x_1,\dots,x_n) = \sum_{1\leq j_1<j_2<\dots<j_d \leq n} x_{j_1} x_{j_2}\cdots x_{j_d}.
\end{equation*} By abuse of notation, we let $S_d(A)$ be the value of $S_d$ at a point whose coordinate is any permutation of $A$. We consider the following polynomial in one complex variable
\begin{equation*}
  F_A(z) = z^n - S_1(A) \cdot z^{n-1} + S_2(A) \cdot z^{n-2} - \cdots + (-1)^{n-1} S_{n-1}(A) \cdot z.
\end{equation*} A \textit{critical point} of $F_A$ refers to a number $p$ such that $F'_A(p) = 0$. We define a \textit{critical value} of $F_A$ to be a number $v$ such that $F_A(p) + (-1)^n v = 0$ for some critical point $p$. By elementary algebra, $v$ is a critical value if and only if the equation $F_A(z) + (-1)^n v = 0$ has a multiple root. By elementary calculus, $F_A$ actually has $n-1$ distinct real critical points interlaced between the consecutive elements in $A$. Hence, all critical values of $F_A$ are real numbers, though not necessarily distinct. The importance of complex numbers will be evident in the proof of our criterion.

\begin{theorem}\label{general}
  Assume that $F_A$ has $n-1$ distinct critical values. Then every affine hyperplane $H\subset \mathbb{R}^n$ with $H\neq H_A$ contains at most $(n-1)!$ vertices of $P_A$. In particular, any vertex cover of $P_A$ has size at least $n$.
\end{theorem}

\noindent As a consequence, $P_A$ generated by a generic $A$ would require at least $n$ elements in its vertex cover, see Proposition~1.6 in \cite{huang2021fraction} for a different proof of this fact. As another consequence, when $A = \{a_1<a_2<a_3<a_4\}$, we can easily argue that the existence of a size-three vertex cover of $P_A$ implies $a_1+a_4 = a_2+a_3$ by applying Theorem~\ref{general}. We first deduce Theorem~\ref{thm:main}.

\begin{proof}[Proof of Theorem~\ref{thm:main}]

  Let $n \geq 3$ be odd and write $F_{n} = F_{[n]}$. According to Theorem~\ref{general}, it suffices to verify that $F_n$ has $n-1$ distinct critical values. We consider the real polynomial $G_n(x) = \prod_{j\in [n]} (x-j)$ and notice that $G_n(x) = F_n(x) - n!$. Hence, the critical points of $G_n$ coincide with those of $F_n$, and it suffices to prove $G_n$ has $n-1$ distinct critical values.

  Now, let $v_j$ be the extreme value of $G_n$ on the interval $(j,j+1)$ for $j\in [n-1]$. Notice that $v_1,\dots,v_{n-1}$ are the critical values of $G_n$. Since $n$ is odd, the graph of the function $G_n$ is symmetric with respect to the point $\frac{n+1}{2}$ on the $x$-axis. Therefore, it suffices to prove $|v_{j}| < |v_{j+1}|$ for every integer $\frac{n}{2} < j < n$. Let $\delta$ be defined by $v_j = G_n(j + 1 - \delta)$, we can compute
  \begin{align*}
    \frac{|G_n(j+1 - \delta)|}{|G_n(j+1 + \delta)|} &= \frac{(j-\delta)(j-1-\delta)\cdots(1-\delta)\cdot \delta(1+\delta)\cdots(n-j-1+\delta)}{(j+\delta)(j-1+\delta)\cdots \delta \cdot (1-\delta)(2-\delta)\cdots(n-j-1-\delta)}\\
    &=\frac{(n-j-\delta)(n-j+1-\delta)\cdots(j-\delta)}{(n-j+\delta)(n-j+1+\delta)\cdots(j+\delta)} < 1.
  \end{align*} This implies $|v_{j}| = |G_n(j+1 - \delta)| < |G_n(j+1 + \delta)| \leq |v_{j+1}|$ as wanted.
\end{proof}

Our proof of Theorem~\ref{general} requires surface-level knowledge of complex algebraic geometry. We refer the reader to the first section of \cite{griffiths1978principles} and the first chapter of \cite{forster1981lectures} for an introduction. Given a covering map $f: X \to Y$ and a base point $y \in Y$, the \textit{monodromy action} of the fundamental group $\pi_1(Y,y)$ on the fiber $f^{-1}(y)$ is as follows: we can always lift a loop representing $\ell \in \pi_1(Y,y)$ starting at a point $x \in f^{-1}(y)$; the image of $x$ under the action of $\ell$ is the ending point of this lifting, and the monodromy theorem guarantees this ending point is well-defined. To avoid wordiness in our proof, ``near'' a point means ``on some neighborhood of'' that point.

\begin{proof}[Proof of Theorem~\ref{general}]

  Let $V \subset \mathbb{C}^n$ be the algebraic variety (not necessarily irreducible) defined as the set of common zeroes of the polynomials $S_d(\textbf{x}) - S_d(A)$ for $d \in [n-1]$. Note that all vertices of $P_A$ are on $V$. It suffices to prove that $V$ is irreducible and one-dimensional (over $\mathbb{C}$). If this is the case, the degree of the curve $V$ will be at most the product of the degrees of the defining polynomials, which is $(n-1)!$. A hyperplane $H$ in $\mathbb{R}^n$ can also be regarded as a hyperplane in $\mathbb{C}^n$. Given that $V$ is an irreducible curve, $V \cap H$ either equals $V$ or has dimension zero. In the former case, all vertices of $P_A$ are contained in $H$, which implies $H = H_A$. In the latter case, we have $|V \cap H| \leq \deg(V)\leq (n-1)!$ by B\'ezout's theorem.

  We consider the holomorphic function $f: V \to \mathbb{C}$ such that $f(x_1,x_2,\dots,x_n) = x_1x_2\dots x_n$. Let $R_y$ be the collection of $n$ roots (possibly repeated) of the equation $F_A(z) + (-1)^{n} y = 0$ for fixed $y \in \mathbb{C}$. Importantly, the preimage $f^{-1}(y)$ consists of all permutations of $R_y$. In particular, $|f^{-1}(y)| < \infty$ for all $y \in \mathbb{C}$. This means any irreducible component of $V$ has dimension at most one. Because the roots of a polynomial vary continuously as a function of the coefficients (see e.g. \cite{harris1987shorter}), $V$ does not have isolated points, so $V$ is purely one-dimensional. It suffices to prove $V$ is irreducible.

  Now, we write $\Omega = \mathbb{C} \setminus \{v_1,v_2,\dots,v_{n-1}\}$ with $v_1,\dots,v_{n-1}$ being the critical values of $F_A$. Let $U = f^{-1}(\Omega)$. We can regard $f: U \to \Omega$ as a covering map between Riemann surfaces (not necessarily connected) as follows: for a point $\textbf{x} \in U$, there are holomorphic functions $x_1,\dots,x_n$ defined near $f(\textbf{x})$ with $(x_1(f(\textbf{x})),\dots, x_n(f(\textbf{x}))) = \textbf{x}$; moreover, $F_A(x_j(y)) + (-1)^n y = 0$ for all $y$ near $f(\textbf{x})$, see e.g. Corollary~8.8 in~\cite{forster1981lectures}; then the holomorphic mapping $(x_1(y),\dots,x_n(y))$ near $f(\textbf{x})$ is a parametrization of $U$ near $\textbf{x}$. Since $V$ is purely one-dimensional and $U$ differs from $V$ by finitely many points, $U$ being irreducible implies $V$ being irreducible. Since $U$ is smooth by its parametrization, its irreducibility is implied by its connectedness. As $\Omega$ is path-connected, we only need to show points in $f^{-1}(y)$ are connected by paths in $U$ for some $y \in \Omega$.

  It suffices to prove the monodromy action of $\pi_1(\Omega,y)$ on the fiber $f^{-1}(y)$ is transitive for some $y\in \Omega$. Note that a permutation of $R_y$ naturally acts on $f^{-1}(y)$ by permuting the coordinates. We have the following two claims.

  \begin{claim}\label{transposition}
    For $y \in \Omega$ and $j \in [n-1]$, let $\ell_j \in \pi_1(\Omega,y)$ be represented by a loop at $y$ whose winding number equals 1 around $v_j$, and equals 0 around $v_k$ for $k\neq j$. Then the monodromy action of $\ell_j$ on $f^{-1}(y)$ is the same as a transposition in the permutation group of $R_y$.
  \end{claim}

  \begin{claim}\label{cycle}
    For a complex number $y$ with sufficiently large $|y|$, let $\ell_n \in \pi_1(\Omega,y)$ be represented by a loop at $y$ whose trajectory is a circle centered at origin. Then the monodromy action of $\ell_n$ on $f^{-1}(y)$ is the same as a cycle of length $n$ in the permutation group of $R_y$.
  \end{claim}

  We choose the loops such that $\ell_1\cdot \ell_2\cdots \ell_{n-1}=\ell_n$ in $\pi_1(\Omega,y)$ as in Figure~\ref{fig}.
  Let $\tau_j\in S_n$ be the permutation of the $n$ roots $R_y$ induced by the monodromy along $\ell_j$. By Claims~\ref{transposition} and \ref{cycle}, each $\tau_j$ ($1\le j\le n-1$) is a transposition and $\tau_n$ is an $n$-cycle, and the relation above implies $\tau_1\cdots \tau_{n-1}=\tau_n$.

  Now form a graph $\Gamma$ on the vertex set $R_y$ by putting an edge between $a$ and $b$ whenever $\tau_j=(ab)$ for some $1\le j\le n-1$. Each transposition $\tau_j$ preserves every connected component of $\Gamma$ setwise, hence so does the subgroup $G=\langle\tau_1,\dots,\tau_{n-1}\rangle$. In particular, $\tau_n\in G$ preserves each connected component. If $\Gamma$ were disconnected, then $\tau_n$ would preserve a proper nonempty subset of $R_y$, contradicting that $\tau_n$ is an $n$-cycle. Therefore $\Gamma$ is connected, so by Lemma~3.10.1 in \cite{GodsilRoyle2001AGT} the transpositions $\tau_1,\dots,\tau_{n-1}$ generate the whole symmetric group $S_n$. Hence, the monodromy action of $\pi_1(\Omega,y)$ on $f^{-1}(y)$ is transitive as wanted.

  \begin{figure}[ht]
    \centering
    \includegraphics{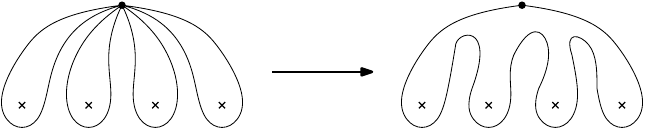}
    \caption{$\ell_1 \cdot \ell_2 \cdots \ell_{n-1} = \ell_n$ in the fundamental group.}
    \label{fig}
  \end{figure}

  Next, we give a proof for Claim~\ref{transposition}. By our hypothesis, $F_A$ has $n-1$ distinct critical points $p_1,\dots,p_{n-1}$. Crucially, every $p_j$ must be a simple critical point, that is, $F_A''(p_j) \neq 0$. We consider the multi-variable complex function $\alpha(z,y) := F_A(z) + (-1)^n y$. We can compute
  \begin{equation*}
    \frac{\partial \alpha}{\partial z}(p_j,v_j) = F'_A(p_j) = 0 \quad\text{and}\quad \frac{\partial^2 \alpha}{\partial z^2}(p_j,v_j) = F''_A(p_j) \neq 0.
  \end{equation*}
  By the Weierstrass preparation theorem, near $(p_j,v_j)$ we can write
  \[
    \alpha(z,y)=\Bigl((z-p_j)^2+\beta_1(y)(z-p_j)+\beta_2(y)\Bigr)\cdot \gamma(z,y),
  \]
  where $\beta_1,\beta_2,\gamma$ are holomorphic, $\beta_1(v_j)=\beta_2(v_j)=0$, and $\gamma(p_j,v_j)\neq 0$.
  Completing the square, set
  \[
    w=(z-p_j)+\frac{\beta_1(y)}{2},
    \qquad
    \beta(y)=\frac{\beta_1(y)^2}{4}-\beta_2(y).
  \]
  Then locally $\alpha(z,y)=\bigl(w^2-\beta(y)\bigr)\cdot \gamma(z,y)$ and we can compute
  \[
    \beta(v_j)=0,
    \qquad \beta'(v_j)=-\frac{(-1)^n}{\gamma(p_j,v_j)}\neq 0.
  \]
  Thus, $\beta$ has a simple zero at $v_j$.
  Fix a basepoint $y_j\in\Omega$ sufficiently close to $v_j$ and let $c_j$ be a small loop around $v_j$ based at $y_j$.
  Near $y_j$, we may choose a holomorphic branch $\delta$ with $\delta^2=\beta$ (see e.g. Lemma~8.7 in \cite{forster1981lectures}). The two local solutions of $\alpha(z,y)=0$ near $p_j$ are given by
  \[
    w_{\pm}(y)=\pm \delta(y),
    \quad\text{equivalently}\quad
    z_\pm(y)=p_j-\frac{\beta_1(y)}{2}\pm \delta(y).
  \]
  Analytic continuation along $c_j$ changes the sign of $\delta$, hence transposes $z_+$ and $z_-$.
  Finally, if $y$ is our global basepoint, choose a path $\gamma_j$ in $\Omega$ from $y$ to $y_j$ and set $\ell_j=\gamma_j c_j \gamma_j^{-1}\in\pi_1(\Omega,y)$.
  The monodromy action of $\ell_j$ is conjugate to that of $c_j$, hence it is again a transposition.

  Finally, we give a proof for Claim~\ref{cycle}. To this end, we consider change of coordinates $s = 1/z$ and $h = 1/y$. Let $\epsilon(s) = (-1)^{n-1}  (1/s^n) (1/F_A(1/s))$. We can remove the singularity of $\epsilon$ at $s=0$ by defining $\epsilon(0) = (-1)^{n-1}$. There is a holomorphic function $\zeta$ defined near $s=0$ with $\zeta^n = \epsilon$. We write $\eta(s) = s \cdot \zeta(s)$ and compute $\eta'(0) \neq 0$, which means $\eta$ is invertible at $s=0$. Observe that $|y| \to \infty$ implies $|z| \to \infty$ under the condition $F_A(z) + (-1)^n y = 0$. Hence, near $h=0$, we have
  \begin{equation*}
    F_A(z) + (-1)^n y = 0 \iff s^n \cdot \epsilon(s) = h \iff (\eta(s))^n = h.
  \end{equation*} For a point $h_\ast$ near but not equal to $h=0$, there is a holomorphic function $\theta$ defined near $h_\ast$ such that $\theta^n = h$, then $s_j(h) = \eta^{-1}(\iota^j \theta(h))$ is a function satisfying $ (\eta(s_j(h)))^n = h$ for $j \in [n]$. Here, $\iota = \exp(2\pi i/n)$ is the $n$-th root of unity. Hence, $1/s_1,\dots,1/s_n$ are the coordinate functions in the parametrization of $U$ near $1/h_\ast$. It is a standard argument that the analytic continuation along a small loop around $h=0$ takes $s_j$ to $s_{j+1}$ and $s_n$ to $s_1$. A small loop around $h=0$ is a large loop around $y=0$, which is homotopic to $\ell_n$.
\end{proof}

\medskip
\noindent {\bf Acknowledgement.} Theorem~\ref{main} was communicated to Ji Zeng by Gyula Károlyi as a conjecture at the 18-th Emléktábla workshop in 2025. We wish to thank Nóra Frankl, Gyula Károlyi, Gergely Kiss, and Gábor Somlai for discussions on this problem. We also wish to thank Zoltán Lóránt Nagy for organizing the workshop. We are grateful to Cosmin Pohoata for informing us of~\cite{pawlowski2024}. We also thank the referee for helpful comments, which improved the exposition and corrected several inaccuracies.

\bibliographystyle{abbrv}
{\footnotesize\bibliography{main}}

\end{document}